\documentclass[12pt,oneside]{amsart}
\usepackage{amsmath}
\usepackage{amsthm}
\usepackage{amsfonts}
\usepackage{amssymb}
\usepackage{enumerate}
\usepackage{comment}
\usepackage{xcolor}
\usepackage{soul}
\usepackage[normalem]{ulem}
\usepackage{hyperref}
\usepackage{cancel}
\usepackage{tikz-cd}
\usepackage{mathrsfs}
\usepackage[colorinlistoftodos,prependcaption]{todonotes}
\newtheorem{theorem}{Theorem}
\newtheorem{lemma}[theorem]{Lemma}

\newtheorem{corollary}[theorem]{Corollary}

\newtheorem{conjecture}[theorem]{Conjecture}
\newtheorem*{conjecture*}{Conjecture}

\usepackage[T1]{fontenc}
\usepackage[utf8]{inputenc}

\theoremstyle{definition}

\theoremstyle{remark}
\newtheorem{remark}[theorem]{Remark}
\newcommand{\DD}{{\mathbb D}}
\newcommand{\OO}{{\mathcal O}}

\newcommand{\FF}{{\mathcal F}}

\newcommand{\RR}{{\mathbb R}}
\newcommand{\CC}{{\mathbb C}}
\newcommand{\cC}{{\mathcal C}}

\newcommand{\Dab}{D_{a,b}}
\newcommand{\co}[1]{\overline{#1}}

\renewcommand{\phi}{\varphi}

\newcommand{\abs}[1]{\lvert#1\rvert}

\subjclass[2020]{32F45}

\begin{document}

\title[Rigidity of Kobayashi isometries]{Rigidity of Kobayashi isometries of a class of 2-dimensional Lempert manifolds}

\address{Doctoral School of Exact and Natural Sciences, Institute of Mathematics, Faculty of Mathematics and Computer Science, Jagiellonian
University,  \L ojasiewicza 6, 30-348 Krak\'ow, Poland}

\author{Anand Chavan}\email{anand.chavan@doctoral.uj.edu.pl}

\thanks{Supported by the Preludium bis grant no. 2021/43/O/ST1/02111 of the National Science Centre,
Poland.}

\keywords{Isometries with respect to invariant distances and metrics, Carathéodory universal set}
\begin{abstract}
   In this article, we study the Kobayashi isometries of 2-dimensional complex manifolds having a finite Carathéodory universal set. In particular, we prove that the Kobayashi isometries of these complex manifolds are (anti)holomorphic.
\end{abstract}
\maketitle
\section{Introduction}
\subsection{Motivation and Results} In the theory of invariant distances and metrics it is an important problem to know, for which complex manifolds an isometry map with respect to their holomorphically invariant distances/metrics, is holomorphic. A simple example of the unit disc shows that every isometry map from the unit disc to itself with respect to the (Poincaré) invariant distance is a holomorphic automorphism of the unit disc or its conjugate, which is indeed (anti)holomorphic (\cite{Jar-Pfl 2013} Proposition 1.1.20). In the case of the bidisc we can find an isometry map from the bidisc to itself with respect to the Kobayashi distance which is not (anti)holomorphic (e.g. $\DD^2\ni(z,w)\mapsto(z,\co{w})$).

The problem of (anti)holomorphic rigidity of the isometry map between complex manifolds has been widely studied. It was first considered in \cite{Kuc-Ray 1988} for the product of $n$-unit Hilbert balls over the Carathéodory metric and subsequently in \cite{Zwo 1993, Zwo 1995} for the Cartesian product of the unit Euclidean ball with respect to the Carathéodory metric or distance.

In the case of the Kobayashi distance in \cite{Gau-Ses 2013}, the (anti)holomorphicity of Kobayashi distance isometries were proved between strongly convex bounded domains of class $\cC^3$ in complex Euclidean space. Additionally, in \cite{Gau-Ses 2017}, the following conjecture was formulated, which remains open to our knowledge (see also \cite{Mah 2012, Bas 2022, Kim-Seo 2022, Bas-Cor 2025}). Although not explicitly formulated in the conjecture, one should consider $M ,M'$ to be equidimensional. 
\begin{conjecture}\label{conj_1}
    Let $M$ and $M'$ be Kobayashi hyperbolic complex manifolds. Let $f:(M,K_M)\to (M',K_{M'})$ be a Kobayashi distance isometry. If $M$ and $M'$ are not biholomorphic to a product of complex manifolds, then $f$ is holomorphic or antiholomorphic.
\end{conjecture}

The most comprehensive result concerning the Kobayashi distance in complex Euclidean bounded domains was established by Edigarian \cite{Edi 2019}, inspired by \cite{Ant 2017}, who proved that the Kobayashi distance isometries between bounded strictly convex domains are (anti)holomorphic. Recent works \cite{Cha-Zwo 2024,Edi 2024} proved the (anti)holomorphicity of Kobayashi metric isometry for the domain diamond and the symmetrized bidisc. The domain diamond is not strictly convex, and the symmetrized bidisc is not biholomorphic to a convex domain.
\subsection{Lempert Theory, C-sets and Carathéodory universal sets.}
For the necessary definitions, see Section \ref{2}. Lempert theory is an effective tool for understanding the isometry problem \cite{Edi 2019, Cha-Zwo 2024, Edi 2024}. Lempert's theorem 
\cite{Lem 1981} states that for a convex or strongly linearly convex domain, the Carathéodory and Kobayashi distances or metrics are equal, viz., there is a unique holomorphically invariant metric or distance on these domains with the distance/metric decreasing property.

In \cite{Kos-Zwo 2021} while studying the norm-preserving extension problem on subsets of the tridisc the authors considered C-sets $\mathcal{M}_{\alpha}$, 
$$\mathcal{M_{\alpha}}:=\{(z_1,z_2,z_3)\in \DD^3:\alpha_1z_1+\alpha_2z_2+\alpha_3z_3=\co{\alpha_1}z_2z_3+\co{\alpha_2}z_1z_3+\co{\alpha_3}z_1z_2\}$$ where $\alpha=(\alpha_1,\alpha_2,\alpha_3)\in \CC^3$
and discovered that the special domain $D_{a,b}\subset \CC^2$ (a domain in $\CC^2$ biholomorphic to $\mathcal{M}_{(a,b,1)}$, where $|a|,|b|,1$ forms sides of a triangle) though not linearly convex admits the equality of these two distances or metrics, adding $D_{a,b}$ to the list of domains where the conclusion of Lempert theorem holds. Moreover, the Carathéodory universal set of $\Dab$ consists of 3 elements, namely $F_{1}(z_1,z_2)=z_1$, $F_{2}(z_1,z_2)=z_2$ and $F_{3}(z_1,z_2)=\frac{az_1+bz_2-z_1z_2}{\co{b}z_2+\co{a}z_{1}-1}$. The other known example of finite number of elements in Carathéodory universal set was for bidisc which has 2 elements, the two coordinate projection function.

In \cite{Agl-Lyk-You 2019}, domains were characterized up to biholomorphism by the cardinality of their Carathéodory universal set. In dimension one, a Lempert domain with a Carathéodory universal set consisting of a single element is biholomorphic to the unit disc, and in this case the Kobayashi isometries are (anti)holomorphically rigid. In dimension two, a Lempert domain whose Carathéodory universal set consists of two elements is biholomorphic to the bidisc, which does not exhibit the (anti)holomorphic rigidity of Kobayashi isometries.

In this article, we prove that a two-dimensional Lempert manifold whose Carathéodory universal set is finite but contains more than two elements exhibits the (anti)holomorphic rigidity phenomenon, in sharp contrast to the bidisc. The following is the main theorem.

\begin{theorem}
    Let $M,N$ be 2-dimensional Lempert manifolds such that their Carathéodory universal set is finite but has at least three elements. Let $F:M\to N$ be a $\cC^1$-smooth Kobayashi metric isometry. Then $F$ is (anti)holomorphic.    
\end{theorem}

As a consequence, we have the (anti)holomorphic rigidity of Kobayashi isometry for the domain $\Dab$ whose Carathéodory universal set consists of 3 elements.
\begin{corollary}
    Let $f:D_{a,b} \to D_{c,d}$ be a $\cC^1$-smooth Kobayashi metric isometry where $(a,b,1),(c,d,1)\in \CC^3$ satisfy the triangle inequality. Then $f$ is (anti)holomorphic. 
\end{corollary}
Above $D_{a,b}=\{z\in \DD^2:|az_1+bz_2-z_1z_2|<|\co{b}z_1+\co{a}z_2-1|\}$. For $(\alpha_1,\alpha_2,\alpha_3)\in \CC^3$ satisfying the triangle inequality, we mean that $|\alpha_1|,|\alpha_2|,|\alpha_3|$ forms the sides of a triangle.

The plan of the paper is to introduce necessary notions in further sections and to develop the proof of this main theorem. The core idea of the proof is to show that the (real)Fr\'echet derivative of the isometry map at a point is (anti)$\CC$-linear and further extends this (anti)$\CC$-linearity for all points of the domain using an argument based on Rado's theorem due to Edigarian \cite{Edi 2019}. 

\begin{remark}
    Regarding Theorem 2, at present we know of only one nontrivial example of a two-dimensional manifold whose Carathéodory universal set consists of three elements. We expect, however, that further two-dimensional manifolds with a finite Carathéodory universal set containing more than two elements do exist, and in such cases Theorem 2 would apply in a much broader range of situations.
\end{remark}
\section{Preliminaries\label{2}}

In this section, we will introduce necessary definitions, give some important results that are true in much more generality, and develop necessary tools to prove the main results. 

\subsection{Invariant distance, indicatrix, affine maps}
 We will recall the definition of the Kobayashi and Carathéodory invariant (pseudo)distance and metric. 

Let $M$ be a complex manifold, consider $z,w \in M$ and $(p;X)\in M \times T_{p}M $. 
    {\it The Carath\'eodory (pseudo)distance} is defined as follows 
$$c_M(z,w):= \sup \{\rho(F(z),F(w)): F \in \mathcal{O}(M,\mathbb{D})\}$$
    where $\rho$ is the Poincar\'e distance on the unit disc $\mathbb D\subset\mathbb C$. 

    {\it The Lempert function} is defined as
    $$l_M(z,w):=\inf\{\rho(\sigma, \zeta): f(\sigma)=z \text{ and } f(\zeta)=w \text{ where }f\in\mathcal{O}(\mathbb{D},M)\}.$$

    {\it The Kobayashi (pseudo)distance $k_M$} is the largest (pseudo)distance not greater than the Lempert function. We have inequalities $c_M\leq k_M\leq l_M$.

    {\it The Kobayashi (pseudo)metric} is defined as
    $$\varkappa_{M}(p;X):=\inf\{|\alpha|: \alpha f'(0)=X \text{ for } f \in \mathcal{O}({\mathbb{D},M})\text{ and } f(0)=p\}.$$
    
    In addition to the above, we define {\it the Caratho\'edory (pseudo)metric} as
    $$\gamma_M(p;X):=\sup\{|F'(p) X|: F\in \mathcal{O}(M,\mathbb{D}) \text{ and } F(p)=0  \}.$$

    Recall that $\gamma_M\leq \varkappa_M$.
    
 We call the complex manifold $M$ {\it Lempert manifold} \cite{Lem 1981} if $M$ is taut and the conclusion of the Lempert Theorem holds on $M$, i.e. if we have the equalities
 $$l_M\equiv k_M\equiv c_M,\; \varkappa_M\equiv \gamma_M.$$
 
The Kobayashi {\it Indicatrix} of a manifold $M$ at a point $p$ is defined as: $$I_{M}(p):=\{X\in T_{p}M:\varkappa_{M}(p;X)<1\}$$ It is important to note that for Lempert manifolds the Kobayashi indicatrix is a convex and balanced, moreover $\varkappa_{M}(p;X)$ is a norm on $T_{p}M$ for $p\in M$ \cite{Jar-Pfl 2013}.

A map $F:M\to N$ where $M,N$ are Lempert manifolds of dimension $m,n$ respectively, is a {\it Kobayashi distance isometry} if $$k_{M}(z,w)=k_{N}(F(z),F(w))$$ for all $z,w\in M$.

A $\cC^1$ smooth map $F:M\to N$ where $M,N$ are Lempert manifolds of dimension $m,n$ respectively, is a {\it Kobayashi metric isometry} if $$\varkappa_{M}(z;X)=\varkappa_{N}(F(z);F'(z)X)$$ for all $z\in M$ and $X \in T_{p}M$. Consequently, $F$ will also be a Kobayashi distance isometry, since the Kobayashi distance is the integrated distance of the Kobayashi metric.

If $F$ is a $\cC^1$ smooth Carathéodory distance isometry between complex manifolds, then $F$ is also a Carathéodory metric isometry. This is due to the fact that $\gamma_M$ is obtained by differentiating $c_M$ (see section 4.3 \cite{Jar-Pfl 2013} for precise meaning).  Moreover, on a Lempert manifold, the Carathéodory distance and metric coincide with the Kobayashi distance and metric, respectively. Therefore, any $\mathcal{C}^1$ smooth Kobayashi distance isometry between Lempert manifolds is also a Kobayashi metric isometry.

When we identify tangent space with complex Euclidean space of appropriate dimensions, $F'(z)$ denotes the real Fr\'echet derivative, that is, $F'(z)$ is an $\RR$-linear map, $F'(z):\CC^{n}(=\RR^{2n})\to \CC^{m}(=\RR^{2m})$.

An $\RR$-linear isomorphism from $\CC^n$ to $\CC^n$ i.e. from $\RR^{2n}$ to $\RR^{2n}$, need not be $\CC$-linear. $\RR$-linear isomorphism $T$ is $\CC$-linear isomorphism iff $$[T][J]=[J][T]$$
where $[T]$ is the matrix representation of the $\RR$-linear map $T$ in the standard real basis and  
$[J] =\begin{bmatrix}
0 & -I \\
I & 0
\end{bmatrix}$, $I$ is the $n\times n$ real identity matrix and is anti $\CC$-linear iff $$[T][J]=-[J][T]$$
In the following lemmas we will see when can an $\RR$ linear isomorphism be a $\CC$ linear isomorphism. 

Before we begin let us introduce some notation. $\partial\DD_{r_i}$ denotes circle of radius $r_{i}>0$ centered at origin in the complex plane. We identify complex coordinates as real coordinates and vice versa as follows: for $z\in \CC^n$, $z=(z_{1},\ldots,z_{n})\leftrightarrow(x_{1},y_{1},\ldots x_{n},y_{n})\in \RR^{2n}$ where $z_{i}=x_i+iy_i$, for $i \in\{1,\dots,n\}$.    
\begin{lemma}\label{4}
    If $A:\CC \to \CC$ is an $\RR$-linear map such that $A(\partial\DD_{r_1})=\partial\DD_{r_2}$ for some $r_1,r_2>0$, then $A(z)=\frac{r_2}{r_1}\omega z$ or $A(z)=\frac{r_2}{r_1}\omega \co{z}$ where $\omega$ is a uni-modular constant.
\end{lemma}
\begin{proof}
    $A:\RR^2\to\RR^2$ is an $\RR$-linear map, i.e $A(x,y)=(a_{1}x+a_{2}y,b_{1}x+b_2y)$ where $a_{i},b_{i}\in \RR$ for $i=1,2$. We can re-write it in terms of complex coordinates as $A(z)=az+b\co{z}$ for some $a,b\in \CC$. 
    
    Since, $A$ maps $\partial\DD_{r_1}$ to $\partial\DD_{r_{2}}$, we have $$|A(r_1e^{i\theta})|=|ar_1e^{i\theta}+br_1e^{-i\theta}|=r_2$$
    (or)
    $$|a|^2r_1^{2}+|b|^2r_1^2+2r_1^{2}Re(a\co{b}e^{2i\theta})=r_2^2$$
    for all $\theta\in \RR$. This implies either $a=0$ or $b=0$.
\end{proof}
A domain $D\subset \CC^n$ is said to be balanced if for any $z\in D$, $\lambda z\in D$ for $\lambda\in \co{\DD}$. For a balanced domain, one can define \emph{the Minkowski functional} $\mu_D : \mathbb{C}^n \to [0,\infty)$ as follows $\mu_D(z)=\inf\{r>0:z\in rD\}=\sup\{|\lambda|:\lambda z\in D\}$. For more details on Minkowski functional we refer reader to \cite{Jar-Pfl 2013}. The following lemma answers the same question but for $\CC^2$.
\begin{lemma}{}\label{5}
    Let $D_1,D_2\subset \CC^2$ be bounded, balanced domains and their Minkowski functional be $\mu_{1}$ and $\mu_{2}$. Consider $T:\CC^2 \to \CC^2$ to be an $\mathbb R$-linear map such that $\mu_1(X)=\mu_{2}(T(X))$ for all $X \in \CC^2$ and $T(L_{1,1}\cup L_{1,2}\cup L_{1,3})=L_{2,1}\cup L_{2,2}\cup L_{2,3}$ where $\{L_{i,j}\}_{j=1}^{3}$ are three different complex lines for $i=1,2$. Then $T$ is $\CC$-linear or anti $\CC$-linear.  
\end{lemma}
\begin{proof}
    $\{L_{1,j}\}_{j=1}^{3}$ and $\{L_{2,j}\}_{j=1}^{3}$ are two sets of three different complex lines. Up to composition with complex linear isomorphisms of $\CC^2$, we may assume that $L_{1,1}=L_{2,1}=\CC \cdot e_1$, $L_{1,2}=L_{2,2}=\CC\cdot e_2$ and $L_{1,3}=\CC \cdot v$, $L_{2,3}= \CC\cdot w$ where $e_1,e_2\in \CC^2$ are standard complex basis vectors of $\CC^2$ and $v=(v_1,v_2),w=(w_1,w_2)\in\CC^2$ such that $v_1,v_2,w_1,w_2$ are all nonzero. As $T$ is a $\RR$ linear map and $T(L_{1,1}\cup L_{1,2}\cup L_{1,3})=L_{2,1}\cup L_{2,2}\cup L_{2,3}$, we may assume (up to composition of $\CC$ linear isomorphism) $T(L_{1,j})=L_{2,j}$, $j=1,2,3$ this makes $T$ an $\RR$ linear isomorphism. In particular, $T(z,0)=(a_1z+a_2\co{z},0)$ and $T(0,z)=(0,b_1z+b_2\co{z})$ where $a_1,a_2,b_1,b_2 \in \CC$

    The domains $D_1, D_2$ are balanced. For $i,j\in\{1,2\}$, $$L_{i,j}\cap D_{i}=\{\lambda e_j\in D_i:\mu_{i}(\lambda e_{j} )\leq 1 \}=\{\lambda e_j\in D_i:|\lambda|\leq 1/\mu_{i}(e_{j} ) \}$$ is a closed disc of radius $1/\mu_{i}(e_{j})$, call it $\co{\DD}_{i,j}$. Since $\mu_1(X)=\mu_{2}(T(X))$, $T$ maps disc $\co{\DD}_{1,1}$ to $\co{\DD}_{2,1}$ and $\co{\DD}_{1,2}$ to $\co{\DD}_{2,2}$ and also their respective boundaries, by Lemma \ref{4}, we have either $a_1=0 \text{ or }a_2=0$ similarly, either $b_1=0 \text{ or }b_2 = 0$. We will show that $T$ is either $\CC$-linear or anti $\CC$-linear.

    To finish the proof, let $T(z_1,0)=(a_2\co{z_1},0)$ and $T(0,z_2)=(0,b_1z_2)$. Since $T(L_{1,3})=L_{2,3}$ we get, in particular, $T(\CC\cdot v)=\CC\cdot w$. $T(v)=(a_2\co{ v_1},b_1 v_2)$, $T(iv)=(-ia_2\co{v_1},ib_1v_2)$ and both belong to $\CC\cdot w$. So, $T(iv)+iT(v)\in \CC\cdot w$. That is $(0,2ib_{2}v_2)\in \CC \cdot w$ which means $b_{2}=0$ or $v_2=0$ is a contradiction.
\end{proof}

\subsection{C-sets and Carathéodory universal sets}
\label{2.2}
Let $M$ be a complex manifold and let $c_{M}$ be its Carathéodory pseudodistance (and $\gamma_{M}$ Carathéodory pseudometric), the
\emph{(infinitesimal) Carathéodory set} or \emph{C-set} in short, is a submanifold $N\subset M$, such that for $z,w\in N$ $$c_{N}(z,w)=c_{M}(z,w)$$ (and $\gamma_{N}(p;X)=\gamma_{M}(p;X)$ for $(p;X)\in N\times T_{p}N$.)

\emph{(Infinitesimal) minimal Carathéodory universal set} $\FF_{M}\subset \OO(M,\DD)$ is the smallest subset of holomorphic maps from $M$ to the unit disc $\DD$, such that for all $z,w\in M$ $$c_{M}(z,w)=\sup_{F\in \FF_{M}}\{\rho(F(z),F(w))\}$$
(and $\gamma_{M}(p;X)=\sup_{F\in\FF_M}\{\gamma_{\DD}(F(p);F'(p)X)\}$ for $(p;X)\in M\times T_{p}M$). Henceforth, the Carathéodory universal set should be considered as a minimal Carathéodory universal set.

Consider a Lempert manifold $M$ with $\FF_M\subset \OO(M,\DD)$ its Carathéodory universal set, since $\varkappa_{M}\equiv \gamma_{M}$ $$\varkappa_{M}(p;X)=\sup_{F\in \mathcal{F}_{M}}\{\gamma_{\mathbb{D}}(F(p);F'(p)X)\}$$
and $\gamma_{\mathbb{D}}(p;X)=\frac{|X|}{1-|p|^2}$. 

When the number of elements in the Carathéodory universal set is finite, we can write the Kobayashi indicatrix of $M$ at a point $p\in M$ as
$$I_{M}(p)=\bigcap_{i}^{\#\FF_M}\{X\in T_{p}M:|L_{i}(X)|<1\}$$ 
where $L_{i}:T_{p}M\to \mathbb{C} $ is a $\mathbb{C}$-linear map, such that for some $F_i\in \FF_{M}$  \[L_{i}(X):=\frac{F'_{i}(p)X}{1-|F_i(p)|^2}\tag{*}\text{\label{*}}\]

For a closed convex set $D\subset \CC^n$, let $p\in \partial D$, a \emph{face} $\mathscr{F}\subsetneq \CC^n$ is the maximal real affine subspace of $\CC^n$ such that $p\in O_{p}\subset  \partial D\cap \mathscr{F} $, where $O_{p}$ is an open subset of $\mathscr{F}$ that contains $p$. By maximal, we mean that if there exists another proper affine subspace, say $V_{0}$ such that $p\in \Tilde{O_p}\subset  \partial D\cap V_0$ then $V_0\subset \mathscr{F}$.

The face is an affine space $\mathscr{F}\subset\CC^n$ so it can be written as $\mathscr{F}=p + \mathcal{L}$ where $\mathcal{L}$ is a linear vector space. We will refer to $\mathcal{L}$ as a \emph{linear face}. It is worth noting that the affine faces (respectively, linear faces) between two compact convex sets are preserved under $\RR$-linear isomorphisms that map the two compact convex sets bijectively.

\begin{lemma}\label{6}
    Let $I:=\bigcap_{i=1}^{k}\{X\in\CC^2:\abs{T_{i}(X)}\leq 1\}$, where $T_{i}:\CC^2\to \CC$ are complex linear maps and  $k\geq3$ is the minimal number that defines $I$. Let $q\in\partial I$ be a boundary point, such that $\abs{T_{i}(q)}=1$ for exactly one $i$, $1\leq i \leq k$ then the faces through $q$ are of the form $\mathscr{F}_{q}=q+Ker(T_{i})$.
\end{lemma}
\begin{proof}
    $I$ is a closed convex set. $\mathscr{F}_{q}$ as defined above, is an affine space containing $q$, we need to show that it is maximal. Let $\Tilde{\mathscr{F}}$ be any other (real)affine space passing through $q$. $$q\in \Tilde{O}_{q}\subset  \Tilde{\mathscr{F}} \cap \partial I$$
    where $\Tilde{O}_{q}$ is an open set in $\Tilde{\mathscr{F}}$.
    Let $v$ be a vector in the linear vector space assigned to the affine space $\Tilde{\mathscr{F}}$. We can find $\epsilon>0$ such that $q+tv\in \Tilde{O}_{q}$, $|t|<\epsilon$. We have $|T_{i}(q+tv)|\leq 1$ such that $\abs{T_{i}(q)}=1$, $|t|<\epsilon$.
    $$|T_{i}(q+tv)|^2=1+2tRe(T_{i}(q)\co{T_i(v)})+t^2|T_{i}(v)|^2\leq 1$$
    $$ 2tRe(T_{i}(q)\co{T_i(v)})+t^2|T_{i}(v)|^2\leq 0$$
    for all $|t|<\epsilon$. In particular $T_{i}(v)=0$ , i.e. $v\in Ker(T_{i})$. Consequently, $$\Tilde{\mathscr{F}}\subset q + Ker(T_{i})=\mathscr{F}_{q}$$
\end{proof}

The face $\mathscr{F}_q$ as above, is an affine vector space. The maximal linear vector space contained in $\mathscr{F}_q$ is a complex line as its linear face $\mathcal{L}_{q}$, which is precisely $\mathcal{L}_{q}=ker(T_i)$ such that $\abs{T_{i}(q)}=1$. 

\section{Proof of the main theorem}

Now we move to the proof of Theorem 2.

\begin{proof}
    $F:M\to N$ is a $\cC^1$-smooth Kobayashi metric isometry, $$\varkappa_{M}(p;X)=\varkappa_{N}(F(p);F'(p)X)$$
    Let $\varkappa_{M}(p;\_):=\mu_{p,M}(\_)$ and $\varkappa_{N}(q;\_):=\mu_{q,N}(\_)$, these are norms on $T_{p}M$ and $T_{q}N$, respectively. 
    $$\mu_{p,M}(X)=\mu_{F(p),N}(F'(p)X)$$ $F'(p):T_{p}M\to T_{F(p)}N$ is an $\mathbb{R}$-linear map, preserving norm. Let $X\in cl(I_{M}(p))$ then $F'(p)X\in cl(I_{N}(F(p)))$. In particular, $F'(p)$ maps $\partial cl(I_{M}(p))$  to $\partial cl(I_{N}(F(p)))$ and also it maps the affine faces to affine faces in the boundaries and corresponding linear faces to linear faces. The boundaries of the indicatrix of both manifolds $M,N$ have at least three affine faces and, respectively, three different complex lines in them. $F'(p)$ being an $\RR$ linear isometry, maps these lines. By the previous Lemma \ref{5}, we find that $F'(p)$ is $\CC$-linear or anti $\CC$-linear, viz. for $p\in M$, $F_{z}(p)=0$ or $F_{\co{z}}(p)=0$. From the following Lemma \ref{lemma:C} and its corresponding Corollary \ref{lemma:C^n} we complete the proof.
\end{proof}

To complete the proof of the main theorem, the following theorems and lemma are applied locally, so we can formulate them on the domains of $\CC^n$.

\begin{theorem}[Rado\cite{Rob 1990}, Chapter Q Corollary 6]
    Let $D\subset \CC^n$ be an open set, if $f:D\to \CC^m$ is a continuous function and is holomorphic in the open subset $D\setminus Z\subseteq D$ where $Z$ is the zero set of $f$, then $f$ is necessarily holomorphic throughout $D$.\label{8}
\end{theorem}

For a $C^1$-smooth mapping $F:D\to\mathbb C^m$, where $D\subset\mathbb C^n$ is a domain, we denote
\begin{equation*}
    F_z(p):=\left [\left(\frac{\partial F_j}{\partial z_k}(p)\right)\right]_{j=1,\ldots,m,k=1,\ldots,n}\in\mathbb C^{m\times n},\; p\in D.
\end{equation*}
Analogously, we denote $F_{\overline{z}}(p)\in\mathbb  C^{m\times n}$.

In our earlier considerations, we were dealing with $\mathbb R$-linear mappings $L:\mathbb C^n\to\mathbb C^m$ that are $\mathbb C$-linear or anti $\mathbb C$-linear.

When dealing with the Fr\'echet derivative $F'(z):\mathbb R^{2n}\to\mathbb R^{2m}$ that is, $\mathbb R$-linear mapping $\mathbb R^{2n}:=\mathbb C^n\to\mathbb R^{2m}:=\mathbb C^m$ the condition of $\mathbb C$-linearity of $F'(z)$ means $F_{\overline z}(p)=0\in\mathbb C^{m\times n}$ whereas the condition of anti $\mathbb C$-linearity means that $F_{z}(p)=0\in\mathbb C^{m\times n}$.

For clarity of the presentation we start with the proof of a result for $m=1$.
\begin{lemma}
\label{lemma:C}
Let $f:D\to \mathbb C$ be a $\mathcal{C}^1$ map such that for any $p\in D\subset\CC^n$ we have the equality $f_z(p)=0$ or $f_{\Bar {z}}(p)=0$. Then $f$ is either holomorphic or anti-holomorphic on $D$. 
\end{lemma}
\begin{proof}
    Define $h:D\to\mathbb C^{n}$ by the formula 
    \begin{equation*}
        h(p):=f_{z}(p)=\left(\frac{\partial f}{\partial z_1}(p),\cdots,\frac{\partial f}{\partial z_n}(p)\right), \;p\in D.
        \end{equation*}
    
    Note that $f$ is continuous on $D$ and holomorphic on an open set $D\setminus Z(h)$, where $Z(h)$ is the zero set of $h$. Since for any $p\in D$ we have $f_{z}(p)=0$ or $f_{\Bar{z}}(p)=0$, we get in  particular that $h(p)=f_z(p)\neq 0$ for  $p\in D\setminus Z(h)$. The assumptions of the lemma imply then that $f_{\Bar{z}}(p)=0$, $p\in D\setminus Z(h)$ or $f$ is holomorphic on $D\setminus Z(h)$. This implies that the mapping $h$ (=$f_z$) is holomorphic on $D\setminus Z(h)$, too. As $h$ is continuous on $D$ we get by Rado's theorem \ref{8} that $h$ is holomorphic on $D$. 
    
    Similarly define the mapping $g:D\to\mathbb C^{n}$ by the formula $g(p)=f_{\Bar{z}}(p)$, $p\in D$. Proceeding as with $h$ we get that $g$ is anti-holomorphic on $D$. 

    In particular, both functions $h$ and $g$ are real analytic on $D$ and for any $p\in D$ we get $h(p)=0$ or $g(p)=0$. Consequently, the identity principle for real analytic function implies that $f_z=h\equiv 0$ on $D$ or $f_{\Bar{z}}=g\equiv 0$ on $D$ and this gives holomorphicity or anti-holomorphicity of $f$ on $D$.  
\end{proof}

\begin{corollary}
\label{lemma:C^n}
Let $f:D\to \mathbb C^m$ be a $\mathcal{C}^1$ map such that for any $p\in D\subset\CC^n$ we have the equality $f_z(p)=0$ or $f_{\Bar {z}}(p)=0$. In other words, for any $p\in D$ if the derivative $f'(p)$ is $\CC$-linear or anti $\CC$-linear then $f$ is either holomorphic or anti-holomorphic, respectively, on $D$.   
\end{corollary}
\begin{proof}
    Define
    \begin{align*}
        A:=&\{p\in D:f_z(p)=0\}=\bigcap_{j=1}^m\{p\in D:(f_j)_z(p)=0\},\\ 
        B:=&\{p\in D:f_{\Bar{z}}(p)=0\}=\bigcap_{j=1}^m\{p\in D:(f_j)_{\Bar{z}}(p)=0\}.
    \end{align*}
To finish the proof, it is sufficient to show that $A$ or $B$ is $D$.
By assumption $A\cup B=D$ and $A$, $B$ are closed in $D$, so one of the sets $A$ or $B$ has non-empty interior. Assume without loss of generality that $A$ has non-empty interior. It follows from the lemma that the mapping $f$ is real analytic, as $f_z\equiv 0$ on $A$ by the identity principle we get $f_z\equiv 0$ on $D$, which shows that $f$ is anti-holomorphic on $D$ and finishes the proof.
\end{proof}

\section{Conclusion and Remarks}
The holomorphic rigidity of Kobayashi isometry in the case of a special 2-dimensional Lempert manifold has been established. However, the methods used here are not enough to conclude the holomorphic rigidity of Kobayashi isometry on general domains in $\CC^n$. The result in this article extends the phenomenon of rigidity previously observed in strictly convex domains \cite{Edi 2019}. The domain Diamond, Symmetrized bidisc, $D_{a,b}$ are Lempert domains, and assuming Conjecture \ref{conj_1} is true at least in the case of equidimensional domains, this raises a natural question on the holomorphic rigidity of Kobayashi isometry on Lempert domains which are not biholomorphic to product domains.

\subsection{Acknowledgment}
I would like to thank my supervisor, Prof. Dr. Hab. Włodzimierz Zwonek, for his continuous support and guidance during the preparation of this article. I am also grateful to Prof. Dr. Hab. Armen Edigarian for valuable discussions and for allowing me to use his ideas in the proofs of the main results presented here.

\end{document}